\documentclass[11pt]{amsart}
\usepackage{amsmath,amssymb,amsthm}
\usepackage{comment}

\renewcommand{\inf}{\operatorname*{inf\vphantom{\sup}}}
\newtheorem{theorem}{Theorem}
\newtheorem{proposition}[theorem]{Proposition}
\newtheorem{lemma}[theorem]{Lemma}
\theoremstyle{remark}
\newtheorem{remark}[theorem]{Remark}
\theoremstyle{definition}

\title{The Big Match under translation-invariant charges}
\author{Abraham Neyman}
\address{Institute of Mathematics and the Federmann Center for the Study of Rationality, The Hebrew University of Jerusalem, Edmund Safra Campus,
Jerusalem 91904, Israel}
\email{aneyman@mail.huji.ac.il}
\thanks{This research was supported in part by Israel Science Foundation
grant 992/23.}  
\date{\today}

\begin{document}

\begin{abstract}
We study finite two-player zero-sum stochastic games in which the
stream of expected stage payoffs is evaluated by a translation-invariant charge on the set of stages. Every patient valuation, that is, every diffuse charge that obeys the time value of money principle, is translation-invariant. By Mertens and Neyman (1981), every finite two-player zero-sum stochastic game has a value; in particular, each player has, for every $\varepsilon>0$, a strategy that is $\varepsilon$-optimal for all patient valuations simultaneously. We show that this robustness breaks down for translation-invariant charges, even when the charge is fixed in advance. In the Big Match with $0$--$1$ payoffs, for every charge, Player~2 can guarantee that Player~1's payoff is at most $1/2$, but cannot guarantee any smaller bound. Nevertheless, we construct a fixed translation-invariant charge such that, for every strategy of Player~1 and every $\delta>0$, Player~2 has a pure reply yielding payoff less than $\delta$. Consequently, the Big Match has no value under this charge. More generally, the supremum of the payoffs that Player~1 can guarantee ranges over $[0,1/2]$ as the
translation-invariant charge varies. This contrasts with finite Markov decision processes, where every stationary Blackwell optimal strategy is optimal for every translation-invariant charge.
\end{abstract}

\maketitle

\section{Introduction}

Flesch, Predtetchinski, Sudderth and Venel \cite{FPSV} study
infinite-horizon Markov decision processes (MDPs) in which the
decision maker evaluates a strategy by aggregating the stream
of expected stage rewards by means of a diffuse charge $\mu$
on the set $\mathbb N$ of stages. They construct a diffuse charge under which no optimal strategy exists, pure or randomized. By contrast, in every finite MDP the decision maker has a pure stationary strategy that is Blackwell optimal \cite{Bla62}; every such strategy is limiting-average optimal (see, e.g., \cite[Ch.~10]{Put94}) and is $\mu$-optimal for every patient valuation, that is, for every diffuse charge $\mu$ that obeys the time value of money principle \cite{Ney23}. Throughout the paper, all MDPs and stochastic games have finite state spaces and finite action sets at each state.

We focus on \emph{translation-invariant} charges, that is, charges satisfying $\mu(E+1)=\mu(E)$ for every $E\subseteq\mathbb N$. Translation-invariant charges are diffuse, and every patient valuation is translation-invariant: by the characterization in \cite{Ney23}, this follows from the fact that for every $E\subseteq\mathbb N$ and every $n$ the sets $E\cap\{1,\dots,n\}$ and $(E+1)\cap\{1,\dots,n\}$ differ in cardinality by at most $1$.
The converse fails: a translation-invariant charge can assign mass one to a set of density zero, and such a charge does not obey the time value of money principle. Nevertheless, Theorem~\ref{thm:mdp} shows that in an MDP every stationary Blackwell optimal strategy is optimal for every translation-invariant charge. It follows from uniform estimates on the total payoff over arbitrary windows of stages, which are known (see, e.g., \cite[Corollary~8.3]{BK78}).

By Mertens and Neyman~\cite{MN81}, every two-player zero-sum stochastic game has a value; in particular, each player has, for every $\varepsilon>0$, a strategy that is $\varepsilon$-optimal for all patient valuations simultaneously.

It is natural to ask whether these simultaneous guarantees extend
to all translation-invariant charges. A first test case is the
Big Match of Blackwell and Ferguson \cite{BF68} with $0$--$1$
payoffs, whose value is $1/2$.

In the Big Match, the stationary strategy of Player~2 that assigns
equal probabilities to $L$ and $R$ yields expected payoff $1/2$
at every stage against every strategy of Player~1. Conversely,
against every strategy of Player~2, Player~1 has a pure strategy
yielding expected payoff at least $1/2$ at every stage.
Consequently,
\[
 \inf_\tau\sup_\sigma \mu(\sigma,\tau)=\frac12
 \qquad\text{for every charge }\mu,
\]
where the infimum ranges over all strategies $\tau$ of Player~2, the
supremum ranges  over all strategies $\sigma$ of Player~1, and $\mu(\sigma,\tau)$ denotes the charge $\mu$ applied to the stream of expected stage payoffs under $(\sigma,\tau)$.

Proposition~\ref{prop:bm} shows that, in sharp contrast both with MDPs under translation-invariant charges and with stochastic games under patient valuations, Player~1 cannot guarantee a positive payoff simultaneously for all translation-invariant charges:
\[
 \sup_\sigma\inf_{\tau,\mu}\mu(\sigma,\tau)=0,
\]
where the supremum ranges over all strategies $\sigma$ of
Player~1 and the infimum ranges over all pairs $(\tau,\mu)$,
where $\tau$ is a strategy of Player~2 and $\mu$ is a
translation-invariant charge. This identity remains valid when the infimum is restricted to pairs of pure strategies of Player~2 and translation-invariant charges.

We strengthen this result by constructing a single
translation-invariant charge $\mu$ such that
\[
 \sup_\sigma\inf_\tau\mu(\sigma,\tau)=0.
\]
Here, too, the infimum may be restricted to pure strategies of
Player~2: for every strategy $\sigma$ of Player~1 and every
$\delta>0$, there is a pure strategy $\tau$ of Player~2 such that
$\mu(\sigma,\tau)<\delta$.

Thus, under this fixed charge, Player~1 cannot guarantee any
positive payoff, whereas Player~2 cannot guarantee a bound
smaller than $1/2$. Consequently, the Big Match has no value
under this translation-invariant charge.

More generally, Theorem~\ref{thm:bm-main} constructs
two translation-invariant charges $\rho$ and $\nu$,
concentrated on disjoint subsets of $\mathbb N$, such that
\[
 \sup_\sigma\inf_\tau
 \bigl(\alpha\rho+(1-\alpha)\nu\bigr)(\sigma,\tau)
 =
 \frac{\alpha}{1+\alpha},
 \qquad 0\leq\alpha\leq1.
\]
The same identity holds when the infimum is restricted
to pure strategies of Player~2. Thus, for every $v\in[0,1/2]$, there is a
translation-invariant charge under which the supremum
of the payoffs that Player~1 can guarantee is $v$.

\section{Charges and ultrafilters}
\label{bm:sec:charges-ultrafilters}

\paragraph{Charges and translation invariance.}
A \emph{charge} on $\mathbb N=\{1,2,\ldots\}$ is a finitely additive
probability $\mu:2^{\mathbb N}\to[0,1]$: $\mu(\mathbb N)=1$ and
$\mu(E\cup F)=\mu(E)+\mu(F)$ for disjoint $E,F\subseteq\mathbb N$.
For a bounded real sequence $g=(g_t)_{t\ge1}$, write
$\mu(g)=\int_{\mathbb N}g_t\,d\mu(t)$.
The integral is defined on simple functions by finite additivity and
extended by uniform approximation. It is positive and linear, satisfies
$\mu(\mathbf 1)=1$, and obeys
\[
 |\mu(g)-\mu(h)|\le \|g-h\|_\infty.
\]

A charge is \emph{diffuse} if it assigns mass zero to every finite set.
Then $\mu(g)=\ell$ whenever $g_t\to\ell$; more generally,
$\mu(g)=\mu(h)$ whenever $g_t-h_t\to0$.
It is \emph{translation-invariant} if
\[
 \mu(E+1)=\mu(E)\quad(E\subseteq\mathbb N),
 \qquad E+1=\{t+1:t\in E\}.
\]
Every translation-invariant charge is diffuse: its singleton masses are
equal, and finite additivity forces their common value to be zero.
Translation invariance is equivalently expressed as $\mu(Sg)=\mu(g)$
for every bounded $g$, where $(Sg)_t=g_{t+1}$.

\paragraph{Ultrafilters.}
An \emph{ultrafilter} on $\mathbb N$ is a family $\mathcal U$ of subsets of $\mathbb N$ that is closed under supersets and finite intersections, and contains exactly one of $E$ and $\mathbb N\setminus E$ for every
$E\subseteq\mathbb N$. If finitely many sets cover $\mathbb N$, at least one of them belongs to $\mathcal U$.

The associated set function
\[
 u_{\mathcal U}(E)=
 \begin{cases}
  1,&E\in\mathcal U,\\
  0,&E\notin\mathcal U
 \end{cases}
\]
is a $\{0,1\}$-valued charge, and every such charge arises this way. Thus an ultrafilter can be viewed as the family of full-mass sets of a $\{0,1\}$-valued charge.

A \emph{nonprincipal}, or \emph{free}, ultrafilter on
$\mathbb N$ is an ultrafilter that contains no finite set.
Nonprincipal ultrafilters exist by Zorn's lemma.
Membership in a nonprincipal ultrafilter is unchanged
by finite modifications.

The charge $u_{\mathcal U}$ is diffuse if and only if
$\mathcal U$ is nonprincipal. A diffuse charge need not
be translation-invariant. Indeed, no $\{0,1\}$-valued
probability charge on $\mathbb N$ is translation-invariant,
since translation invariance would assign mass $1/2$
to each of the sets of even and odd integers.

\paragraph{Ultrafilter limits.}
Fix a nonprincipal ultrafilter $\mathcal U$ on $\mathbb N$.
For every bounded real sequence $(x_n)$ there is a unique
real number $\ell$, denoted $\lim_{n\to\mathcal U}x_n$,
such that
\[
 \{n:|x_n-\ell|<\varepsilon\}\in\mathcal U
 \qquad\text{for every }\varepsilon>0.
\]
This limit equals $\int_{\mathbb N}x_n\,du_{\mathcal U}(n)$.
Consequently, ultrafilter limits preserve linear combinations and
inequalities, and agree with ordinary limits whenever these exist.
In particular,
\[
 \lim_{n\to\mathcal U}\mathbf 1_E(n)=u_{\mathcal U}(E),
 \qquad
 \liminf_{n\to\infty}x_n
 \le \lim_{n\to\mathcal U}x_n
 \le \limsup_{n\to\infty}x_n.
\]

If $(\lambda_n)$ is a sequence of charges, then
\[
 \lambda(E)=\lim_{n\to\mathcal U}\lambda_n(E)
 \qquad(E\subseteq\mathbb N)
\]
defines a charge, and for every fixed bounded sequence $g$,
\[
 \lambda(g)=\lim_{n\to\mathcal U}\lambda_n(g).
\]
The identity holds first for simple functions and then, by uniform
approximation, for all bounded $g$.
If every $\lambda_n$ is translation-invariant, so is $\lambda$.

\paragraph{Interval averages.}
Let $(J_n)$ be finite nonempty intervals in $\mathbb N$
with $|J_n|\to\infty$. Then
\[
 \mu_J(E)=\lim_{n\to\mathcal U}
             \frac{|E\cap J_n|}{|J_n|}
\]
defines a translation-invariant charge, since shifting $E$ changes
each fraction by at most $1/|J_n|$.

For every bounded $g$,
\[
 \mu_J(g)=\lim_{n\to\mathcal U}
             \frac{1}{|J_n|}\sum_{t\in J_n}g_t.
\]
Taking $J_n=\{1,\ldots,n\}$ gives a charge that agrees with the
ordinary Ces\`aro limit whenever it exists.
An arbitrary translation-invariant charge need not have this property:
a set of natural density zero that contains arbitrarily long intervals
can have mass one under a charge constructed from those intervals.

\paragraph{Iterated ultrafilter limits.}
The construction of a charge by an ultrafilter limit can be
applied repeatedly. If $(\lambda_{k,n})_{k,n\ge1}$ is a family
of charges and $\mathcal U_1$ and $\mathcal U_2$ are fixed nonprincipal ultrafilters,
then
\[
 \lambda(E)
   =\lim_{k\to\mathcal U_1}\lim_{n\to\mathcal U_2}
        \lambda_{k,n}(E),
 \qquad E\subseteq\mathbb N,
\]
defines a charge. Here the limit in $n$ is taken first, for
each fixed $k$, and the limit in $k$ is then taken. Applying
the preceding integral identity twice gives, for every fixed
bounded sequence $g$,
\[
 \lambda(g)
   =\lim_{k\to\mathcal U_1}\lim_{n\to\mathcal U_2}
        \lambda_{k,n}(g).
\]
The same observations apply to any finite number of successive
ultrafilter limits. The ultrafilters $\mathcal U_1$ and $\mathcal U_2$ need not be distinct; in the construction below we take
$\mathcal U_1=\mathcal U_2=\mathcal U$.

For background, see \cite[Sections~2 and~4]{Galvin2009} or
\cite[Sections~1.1 and~3.1]{Goldbring2022}.

\begin{lemma}\label{lem:inv}
Let $\mu$ be a translation-invariant charge. Then, for every bounded $g$ and all $N,k\in\mathbb N$,
  \[
  \inf_{t\ge k}\frac1N\sum_{j=0}^{N-1}g_{t+j}
  \;\le\; \int g\,d\mu \;\le\;
  \sup_{t\ge k}\frac1N\sum_{j=0}^{N-1}g_{t+j}.
  \]
\end{lemma}

\begin{proof}
Let $\varphi_N(t)=\frac1N\sum_{j=0}^{N-1}g_{t+j}$. By the equality $\int Sg\,d\mu=\int g\,d\mu$,
$\int g\,d\mu=\int\varphi_N\,d\mu$. Moreover
$\varphi_N\le \sup_{t\ge k}\varphi_N(t)+2\|g\|_\infty\mathbf 1_{\{1,\dots,k-1\}}$,
and $\mu(\{1,\dots,k-1\})=0$. This gives the upper bound; the lower bound
is symmetric.
\end{proof}

\begin{lemma}\label{lem:windows}
Let $I\subseteq\mathbb N$ be infinite and, for each $N\in I$, let
$W_N=\{k_N,k_N+1,\dots,k_N+N-1\}$ be a window of length $N$. Then there is a translation-invariant charge $\mu$ such that, for every bounded $g$,
\[
\liminf_{\substack{N\to\infty\\ N\in I}}\frac1N\sum_{t\in W_N}g_t
\;\le\;\int g\,d\mu\;\le\;
\limsup_{\substack{N\to\infty\\ N\in I}}\frac1N\sum_{t\in W_N}g_t .
\]
\end{lemma}

\begin{proof}
Enumerate $I$ increasingly as $N_1<N_2<\cdots$ and apply the interval-average construction to $J_j=W_{N_j}$. It gives a translation-invariant charge $\mu$ such that, for every bounded $g$,
\[
 \mu(g)=\lim_{j\to\mathcal U} \frac1{N_j}\sum_{t\in W_{N_j}}g_t.
\]
The asserted inequalities follow because a nonprincipal
ultrafilter limit lies between the ordinary $\liminf$
and $\limsup$.
\end{proof}

\section{Finite MDPs}

Consider an MDP with finite state space $\mathcal S$, finite action sets
$A(s)$, rewards $r(s,a)$, transition probabilities $p(\cdot\mid s,a)$, and
initial state $s_1$. For a strategy $\sigma$ let
$g^\sigma_t=\mathbb E_\sigma[r(s_t,a_t)]$, and for a charge $\mu$ set
$\mu(\sigma)=\int g^\sigma\,d\mu$. Write $\|r\|=\max_{s,a}|r(s,a)|$.

For $\lambda\in(0,1]$, let $v_\lambda(s)$ denote the normalized
$\lambda$-discounted value of the MDP with initial state $s$, that is, the
supremum over strategies $\sigma$ of
$\mathbb E_\sigma\bigl[\sum_{t\ge1}\lambda(1-\lambda)^{t-1}r(s_t,a_t)\bigr]$.
Then $\|v_\lambda\|_\infty\le\|r\|$, and for every $s$,
\begin{equation}\label{eq:bellman}
v_\lambda(s)=\max_{a\in A(s)}\Bigl[\lambda r(s,a)
+(1-\lambda)\sum_{s'}p(s'\mid s,a)v_\lambda(s')\Bigr].
\end{equation}
A stationary strategy $f$, which plays a mixed action $f(s)\in\Delta(A(s))$ at each state $s$, is \emph{Blackwell optimal} if there is $\lambda_0>0$ such that $f$ is optimal in the $\lambda$-discounted MDP, for every initial state, for every $\lambda\in(0,\lambda_0]$. Blackwell \cite{Bla62} proved that every
MDP has a pure stationary Blackwell optimal strategy. For a stationary $f$, let $P_f$ be the $\mathcal S\times\mathcal S$ transition matrix and $r_f$ the reward vector that it induces:
$P_f(s,s')=\sum_{a\in A(s)}f(s)[a]\,p(s'\mid s,a)$ and
$r_f(s)=\sum_{a\in A(s)}f(s)[a]\,r(s,a)$. Then $g^f_t=(P_f^{t-1}r_f)(s_1)$.

\begin{lemma}\label{lem:opt}
Let $f$ be a stationary Blackwell optimal strategy. Then:
\begin{enumerate}
  \item the limits $v=\lim_{\lambda\to0}v_\lambda$ and
  $h=\lim_{\lambda\to0}(v_\lambda-v)/\lambda$ exist;
  \item $P_fv=v$ and $r_f=v+h-P_fh$;
  \item for every $s$ and $a\in A(s)$,
  $\sum_{s'}p(s'\mid s,a)v(s')\le v(s)$, and if equality holds, then
  $r(s,a)+\sum_{s'}p(s'\mid s,a)h(s')\le v(s)+h(s)$;
  \item there is a function $H$ on $\mathcal S$ such that, for every $s$ and
  $a\in A(s)$,
  \[
  r(s,a)+\sum_{s'}p(s'\mid s,a)H(s')\le v(s)+H(s).
  \]
\end{enumerate}
\end{lemma}

\begin{proof}
(1) For $\lambda\in(0,\lambda_0]$, the optimality of $f$ gives
\begin{equation}\label{eq:fixedf}
v_\lambda=\lambda r_f+(1-\lambda)P_fv_\lambda ,
\end{equation}
so $v_\lambda=\lambda\bigl(I-(1-\lambda)P_f\bigr)^{-1}r_f$. The entries of
this vector are rational functions of $\lambda$ that are bounded by $\|r\|$ on $(0,\lambda_0]$. Hence they have no pole at $0$, and so they are analytic at $0$. This proves~(1), and moreover $v_\lambda=v+\lambda h+O(\lambda^2)$ as $\lambda\to0$.

(2) Substituting this expansion in equation~\eqref{eq:fixedf} and comparing the coefficients of $1$ and $\lambda$ gives $v=P_fv$ and $h=r_f-P_fv+P_fh$, that is, $r_f=v+h-P_fh$.

(3) By the optimality equation~\eqref{eq:bellman},
$v_\lambda(s)\ge\lambda r(s,a)+(1-\lambda)\sum_{s'}p(s'\mid s,a)v_\lambda(s')$.
Substituting the expansion and comparing the coefficients of $1$ gives the
first inequality. If it holds with equality, comparing the coefficients of
$\lambda$ gives the second.

(4) Let $\gamma>0$ be the minimum of $v(s)-\sum_{s'}p(s'\mid s,a)v(s')$ over the pairs $(s,a)$ for which it is positive (any $\gamma>0$ if there is no such pair). Choose $M\ge0$ with
$M\gamma\ge\max_{s,a}\bigl[r(s,a)+\sum_{s'}p(s'\mid s,a)h(s')-v(s)-h(s)\bigr]$,
and set $H=h+Mv$. If $\sum_{s'}p(s'\mid s,a)v(s')=v(s)$, the claimed
inequality is the second inequality in~(3) plus $M$ times this equality.
Otherwise, $M\sum_{s'}p(s'\mid s,a)v(s')\le Mv(s)-M\gamma$, so the left side is at most $r(s,a)+\sum_{s'}p(s'\mid s,a)h(s')+Mv(s)-M\gamma
\le v(s)+h(s)+Mv(s)$.
\end{proof}

The following estimates are essentially known. Let $V_N(s)$ denote the
maximal expected total payoff in the $N$-stage MDP with initial state $s$.
Then $V_N=Nv+O(1)$, and every stationary Blackwell optimal strategy $f$ is uniformly $N$-stage optimal, that is, $V_N(s)-\mathbb E_{f}\bigl[\sum_{t=1}^Nr(s_t,a_t)\,\big|\,s_1=s\bigr]$ is
bounded in $N$ and $s$; see, e.g., \cite[Corollary~8.3]{BK78} (where Blackwell optimal strategies are called uniformly discount optimal) and \cite{Bro65}. Conditioning on the history up to stage $k$, and using
Lemma~\ref{lem:opt}(2)--(3), these two facts give both bounds of
Lemma~\ref{lem:mdpwindows} on every window. For completeness, we give a
short self-contained proof.

\begin{lemma}\label{lem:mdpwindows}
Let $f$ be a stationary Blackwell optimal strategy, and let $v$ be as in Lemma~\ref{lem:opt}. There is a constant $C$ such that, for every initial state, every strategy $\sigma$, and all $k,N\in\mathbb N$,
\[
\Bigl|\sum_{t=k}^{k+N-1}g^f_t-Nv(s_1)\Bigr|\le C
\qquad\text{and}\qquad
\sum_{t=k}^{k+N-1}g^\sigma_t\le Nv(s_1)+C .
\]
\end{lemma}

In words: on every window of $N$ consecutive stages, wherever it starts, the total payoff of $f$ differs from $Nv(s_1)$ by a bounded amount, and no
strategy exceeds $Nv(s_1)$ by more than a bounded amount. The bound does not depend on the starting stage of the window, its length, or the strategy.

\begin{proof}
Let $h$ and $H$ be as in Lemma~\ref{lem:opt}, and let
$C=2\max\bigl(\|h\|_\infty,\|H\|_\infty\bigr)$.

\medskip\noindent\textbf{The bound for $f$.}
By Lemma~\ref{lem:opt}(2), $P_f^{t-1}v=v$ and
$g^f_t=\bigl(P_f^{t-1}(v+h-P_fh)\bigr)(s_1)
=v(s_1)+(P_f^{t-1}h)(s_1)-(P_f^{t}h)(s_1)$. Summing over the window, the
$h$-terms telescope:
\[
\sum_{t=k}^{k+N-1}g^f_t=Nv(s_1)+(P_f^{k-1}h)(s_1)-(P_f^{k+N-1}h)(s_1),
\]
and the last two terms are at most $2\|h\|_\infty$ in absolute value.

\medskip\noindent\textbf{The bound for $\sigma$.}
Fix any strategy $\sigma$. Conditioning on the history up to $(s_t,a_t)$,
Lemma~\ref{lem:opt}(4) gives
\[
g^\sigma_t\le\mathbb E_\sigma\bigl[v(s_t)+H(s_t)-H(s_{t+1})\bigr],
\]
and the first inequality in Lemma~\ref{lem:opt}(3) gives
$\mathbb E_\sigma[v(s_{t+1})]\le\mathbb E_\sigma[v(s_t)]$, so
$\mathbb E_\sigma[v(s_t)]\le v(s_1)$ for all $t$. Summing over the window, the
$H$-terms telescope and
\[
\sum_{t=k}^{k+N-1}g^\sigma_t\le Nv(s_1)+2\|H\|_\infty .
\]
\end{proof}

Taking $k=1$, Lemma~\ref{lem:mdpwindows} shows that $v(s_1)$ is the
limiting-average value of the MDP and that $f$ is limiting-average optimal.
The rate $O(1/N)$ for window averages cannot be improved in general: in a
deterministic MDP that alternates between two states with rewards $1$ and $0$,
the average over any window of odd length $N$ differs from $1/2$ by $1/(2N)$.

\begin{theorem}\label{thm:mdp}
Let $f$ be a stationary Blackwell optimal strategy, and let $v$ be the
limiting-average value of the MDP. For every initial state, every strategy
$\sigma$, and every translation-invariant charge $\mu$,
\[
\mu(f)=v(s_1)\ge\mu(\sigma).
\]
\end{theorem}

\begin{proof}
By Lemma~\ref{lem:mdpwindows}, for every $N$ all window averages of
$g^f$ of length $N$ lie within $C/N$ of $v(s_1)$, and all window averages of $g^\sigma$ of length $N$ are at most $v(s_1)+C/N$. By
Lemma~\ref{lem:inv}, $|\mu(f)-v(s_1)|\le C/N$ and
$\mu(\sigma)\le v(s_1)+C/N$ for every $N$. Letting $N\to\infty$ gives $\mu(f)=v(s_1)\ge\mu(\sigma)$.
\end{proof}

\begin{remark}
Theorem~\ref{thm:mdp} can also be deduced from the proof of
Lemma~\ref{lem:mdpwindows}. That proof shows that $g^f=v(s_1)+\varphi-S\varphi$ with
$\varphi_t=(P_f^{t-1}h)(s_1)$, and $g^\sigma\le w+\psi-S\psi$ with $w_t=\mathbb E_\sigma[v(s_t)]\le v(s_1)$ and
$\psi_t=\mathbb E_\sigma[H(s_t)]$. Since $\varphi$ and $\psi$ are bounded,
$\int(\varphi-S\varphi)\,d\mu
=\int(\psi-S\psi)\,d\mu=0$. It follows that $\mu(f)=v(s_1)$ and
$\mu(\sigma)\le\mu(w)\le v(s_1)$.
\end{remark}

\begin{remark}\label{rem:stationary}
Theorem~\ref{thm:mdp} extends to every two-player zero-sum stochastic game in which both players have stationary optimal strategies in the limiting-average sense. Indeed, if Player~1 plays a stationary strategy $x$, then Player~2 faces an MDP; let $w(x)$ denote its limiting-average value at $s_1$, with Player~2 minimizing.
Theorem~\ref{thm:mdp}, applied to this MDP, gives $\mu(x,\tau)\ge w(x)$ for every strategy $\tau$ of Player~2 and every translation-invariant charge $\mu$, with equality when $\tau$ is a stationary Blackwell optimal strategy of this MDP. Hence $\inf_\tau\mu(x,\tau)=w(x)$ for every translation-invariant charge $\mu$.
If $x$ is limiting-average optimal, then $w(x)\ge v(s_1)$, where $v$ is the limiting-average value of the game. The symmetric argument applies to a stationary limiting-average optimal strategy $y$ of Player~2. Hence
\[
\mu(x,\tau)\ge v(s_1)\ge\mu(\sigma,y)
\]
for all strategies $\sigma$ and $\tau$ and every translation-invariant charge $\mu$, so $x$ and $y$ are optimal for all translation-invariant charges simultaneously.

Both players have stationary optimal strategies in the
limiting-average sense, for instance, in games with perfect information, in irreducible games, and in games in which one player controls the
transitions; see \cite[Section~17]{BK78} and the references there. More
generally, the same argument shows that if both players have, for every
$\varepsilon>0$, stationary $\varepsilon$-optimal strategies in the
limiting-average sense, then these are $\varepsilon$-optimal for all
translation-invariant charges simultaneously.

The Big Match is the simplest game outside this class: Player~2 has a stationary optimal strategy, but Player~1 has no stationary $\varepsilon$-optimal strategy for $\varepsilon<1/2$.
Proposition~\ref{prop:bm} shows that, correspondingly, Player~1's guarantee with respect to all translation-invariant charges drops to $0$.
\end{remark}

\section{The Big Match}

Consider the Big Match with $0$--$1$ payoffs,
\[
\begin{array}{c|cc}
  & L & R \\ \hline
T & 0 & 1 \\
B & 1^{*} & 0^{*}
\end{array}
\]
where Player~1 (the row player) is the maximizer and
an asterisk denotes absorption at the indicated payoff.
For Player~2, $L$ is the risky action: it leads to the
absorbing payoff $1$ if Player~1 plays $B$.

For strategies $\sigma$ of Player~1 and $\tau$ of Player~2,
let $x_t$ denote the payoff at stage $t$ and, for a
charge $\lambda$, set
\[
 \lambda(\sigma,\tau)
   =\int_{\mathbb N}
       \mathbb E_{\sigma,\tau}[x_t]\,d\lambda(t).
\]
The lower and upper values of the game under $\lambda$
are defined by
\[
 \underline v(\lambda)
   =\sup_\sigma\inf_\tau\lambda(\sigma,\tau),
 \qquad
 \overline v(\lambda)
   =\inf_\tau\sup_\sigma\lambda(\sigma,\tau).
\]
Unless explicitly restricted to pure strategies, the
supremum and infimum range over all strategies of
Player~1 and Player~2, respectively.

We first determine the upper value.

\begin{lemma}
\label{bmfc:lem:upper}
For every charge $\lambda$, the upper value
$\overline v(\lambda)$ is $1/2$.
\end{lemma}

\begin{proof}
The stationary strategy of Player~2 that chooses $L$ and $R$
with equal probabilities yields expected payoff $1/2$ at
every stage against every strategy of Player~1. Hence
$\overline v(\lambda)\le1/2$.

It remains to prove that $\overline v(\lambda)\ge1/2$. Fix a strategy $\tau$ of Player~2 and let $q_t$ be the probability of $L$ at stage $t$ when Player~1 always plays $T$. If $q_t\le1/2$ for every $t$, always playing $T$ yields expected payoff $1-q_t\ge1/2$ at every stage.
Otherwise, let $t_0$ be the first stage with $q_{t_0}>1/2$. Playing $T$ before $t_0$ and $B$ at $t_0$ yields expected payoff $1-q_t\ge1/2$ at every stage $t<t_0$ and $q_{t_0}>1/2$ at every stage $t\ge t_0$.

Thus, against every strategy $\tau$ of Player~2, there is a pure strategy $\sigma$ of Player~1 that gives expected payoff at least $1/2$ at every stage. By monotonicity of the integral, $\lambda(\sigma,\tau)\ge1/2$, and hence $\overline v(\lambda)\ge1/2$.
\end{proof}

We first state and prove that Player~1 cannot guarantee
a positive payoff simultaneously for all translation-invariant
charges. Although this conclusion follows from the stronger
fixed-charge result of the next section, whose proof is more
involved, the elementary argument given here explains the
role of long intervals of stages in which Player~2 plays
the risky action $L$.

\begin{proposition}\label{prop:bm}
For every strategy $\sigma$ of Player~1 and every
$\delta>0$, there exist a pure strategy $\tau$ of Player~2
and a translation-invariant charge $\mu$ on $\mathbb N$
such that $\mu(\sigma,\tau)<\delta$.
\end{proposition}

\begin{proof}
Fix a strategy $\sigma$ of Player~1.

\medskip
\noindent\textbf{Step 0: Payoffs against a pure strategy.}
We represent a pure strategy of Player~2 by a sequence
$a=(a_t)_{t\ge1}\in\{L,R\}^{\mathbb N}$, where $a_t$
specifies its action at stage $t$ if absorption has not
occurred earlier.

Let $p_t(a)$ be the probability that Player~1 first
plays $B$ at stage $t$ against this sequence.
The function $p_t(a)$ depends only on
$a_1,\ldots,a_{t-1}$.
Set
\[
 A_1(a)=\sum_{t\ge1}p_t(a)\mathbf1_{\{a_t=L\}},
 \qquad
 A_0(a)=\sum_{t\ge1}p_t(a)\mathbf1_{\{a_t=R\}},
\]
and
\[
 S(a)=1-A_1(a)-A_0(a).
\]
Thus $A_i(a)$ is the probability of absorption with
payoff $i$, for $i\in\{0,1\}$, and $S(a)$ is the
probability that absorption never occurs.
The functions $A_0$, $A_1$, and $S$ are Borel,
since each $p_t$ depends on finitely many coordinates
of $a$.

Write $R(a)=\{t\in\mathbb N:a_t=R\}$.
For every diffuse charge $\lambda$,
\begin{equation}
 \lambda(\sigma,a)=A_1(a)+S(a)\lambda(R(a)).
 \label{bmfc:eq:pure-payoff}
\end{equation}
Indeed, the expected payoff at stage $t$ equals
\[
 \sum_{s\le t}p_s(a)\mathbf1_{\{a_s=L\}}
 +
 \left(1-\sum_{s\le t}p_s(a)\right)
       \mathbf1_{\{a_t=R\}},
\]
and its difference from
$A_1(a)+S(a)\mathbf1_{\{a_t=R\}}$
tends to zero as $t\to\infty$.
Since $\lambda$ is diffuse, the sequence of these
differences has integral zero with respect to $\lambda$,
which proves \eqref{bmfc:eq:pure-payoff}.

\medskip
\noindent\textbf{Step 1: Payoffs against a stationary strategy.}
Fix $p\in(0,1)$, and let $\mathbb P=P_p$ be the product probability on $\{L,R\}^{\mathbb N}$ under which each coordinate equals $L$ with probability $p$. Write $\mathbb E_p$ for expectation with respect to
$\mathbb P$.

Under $P_p$, the conditional probability of $a_t=L$, given $a_1,\ldots,a_{t-1}$, is $p$. Since $p_t(a)$ depends only on these past coordinates,
\[
 \mathbb E_p\!\left[ p_t(a)\mathbf1_{\{a_t=L\}}
   \,\middle|\,a_1,\ldots,a_{t-1}  \right]
 = p\,p_t(a).
\]
Taking expectations and applying the monotone convergence
theorem gives
\[
 \begin{aligned}
 \mathbb E_p A_1
 &= \sum_{t\ge1}\mathbb E_p \bigl[p_t(a)\mathbf1_{\{a_t=L\}}\bigr]\\
 &= p\,\mathbb E_p\sum_{t\ge1}p_t(a) = p(1-\mathbb E_p S).
 \end{aligned}
\]
It follows that
\[
 \mathbb E_p[A_1+pS]=p.
\]

\medskip
\noindent\textbf{Step 2: Long strings of $L$.}
Fix $N\in\mathbb N$. The events
\[
 \{a_{mN+1}=\cdots=a_{(m+1)N}=L\},
 \qquad m=0,1,2,\ldots,
\]
are independent under $\mathbb P$, and each has probability $p^N>0$. By the second Borel--Cantelli lemma, $\mathbb P$-almost surely infinitely many of them occur.

Taking the countable intersection over $N$, we conclude that, with $\mathbb P$-probability one, for every $N\in\mathbb N$ there is a window
$W_N=\{k_N,\ldots,k_N+N-1\}$ such that $a_t=L$
for all $t\in W_N$.

\medskip
\noindent\textbf{Step 3: Choice of $\tau$ and $\mu$.}
Let $\delta>0$. Choose $p\in(0,1)$ with $2p\le\delta$ and apply Steps~1 and~2 with $\mathbb P=P_p$. The equality $\mathbb E_p[A_1+pS]=p$ and the
nonnegativity of $S$ imply that $\mathbb E_p A_1\le p$. Therefore, by Markov's inequality, the event $\{a:A_1(a)<\delta\}$ has $\mathbb P$-probability at least $1/2$.

By Step~2, the event that, for every $N$, there is a window of length $N$ on which $a_t=L$ has $\mathbb P$-probability one. Hence the intersection
of these two events has positive probability.
Choose a sequence $a$ in this intersection and, for each $N$, a corresponding window $W_N$.
Let $\tau$ be the pure strategy that plays $a_t$
at stage $t$ before absorption.

Since $\mathbf1_{R(a)}$ vanishes on every $W_N$, Lemma~\ref{lem:windows}, applied with $I=\mathbb N$ and these windows, gives a translation-invariant charge $\mu$ such that $\mu(R(a))=0$.
As $\mu$ is diffuse, \eqref{bmfc:eq:pure-payoff} gives
\[
 \mu(\sigma,\tau)
   =A_1(a)+S(a)\mu(R(a))
   =A_1(a)<\delta.
\]
This proves the proposition.
\end{proof}

\section{A fixed translation-invariant charge}
\label{sec:bm-fixed-charge}

\begin{theorem}
\label{thm:bm-main}
There exist translation-invariant charges $\rho$ and $\nu$,
concentrated on disjoint subsets of $\mathbb N$, such that,
for every $\alpha\in[0,1]$, the charge
\[
 \mu_\alpha=\alpha\rho+(1-\alpha)\nu
\]
satisfies
\begin{equation}
 \underline v(\mu_\alpha)
 =
 \sup_\sigma\inf_{\tau\ \mathrm{pure}}
       \mu_\alpha(\sigma,\tau)
 =
 \frac{\alpha}{1+\alpha}.
 \label{bmfc:eq:values}
\end{equation}
More precisely, for every strategy $\sigma$ of Player~1
and every $\alpha\in(0,1]$, there is a pure strategy $\tau$
of Player~2 such that
\[
 \mu_\alpha(\sigma,\tau)
 \le\frac{\alpha}{1+\alpha}.
\]
For $\alpha=0$, for every $\sigma$ and every
$\varepsilon>0$, there is a pure $\tau$ such that
$\nu(\sigma,\tau)<\varepsilon$.

Consequently, by Lemma~\ref{bmfc:lem:upper}, the Big Match has no value under $\mu_\alpha$ for $\alpha<1$, and its lower value ranges over $[0,1/2]$ as the translation-invariant charge varies.
\end{theorem}

\subsection*{Two facts about outer probability}

A set has \emph{outer probability one} under a Borel
probability $P$ if it meets every Borel set of positive
$P$-probability. The set itself need not be measurable.

Recall that membership in a nonprincipal ultrafilter
$\mathcal U$ on $\mathbb N$ is unchanged by finite
modifications. Moreover, if finitely many sets cover
$\mathbb N$, at least one of them belongs to $\mathcal U$.

The following lemma extends the classical outer-measure
property of nonprincipal ultrafilters from the fair-coin
product measure to arbitrary probabilities on a finite
alphabet; see \cite[464Ca]{Fremlin} for the classical
statement. We include a proof for completeness.

\begin{lemma}
\label{bmfc:lem:ultrafilter}
Let $F$ be a finite set, let $\pi$ be a probability on $F$,
and let $b\in F$ satisfy $0<\pi(b)<1$.
If $\mathcal U$ is a nonprincipal ultrafilter on
$\mathbb N$, then
\[
 Z_b=\left\{x\in F^{\mathbb N}:
             \{n:x_n=b\}\in\mathcal U\right\}
\]
has outer probability one under
$\pi^{\otimes\mathbb N}$.
\end{lemma}

\begin{proof}
Suppose that $Z_b$ does not have outer probability one. Then $Z_b^c$ contains a Borel set of positive probability and hence, by inner regularity, a compact set $K$ of positive probability.

Let $H$ consist of all finite-coordinate modifications of points of $K$. For each $m$, the set of sequences that agree with some point of $K$ at every coordinate $n>m$ is compact: it is the product of $F^m$ and the
projection of $K$ onto the remaining coordinates.
The union of these sets over $m$ is $H$, so $H$ is Borel.
Since membership in $\mathcal U$ is unchanged by finite modifications, $H\subseteq Z_b^c$.
Moreover, membership in $H$ does not depend on any finite initial segment, so $H$ is a tail event. It has positive probability and therefore has probability one by Kolmogorov's zero-one law.

Put $p=\pi(b)$ and choose an integer $r$ with $rp\ge1$.
Partition $[0,1)$ into half-open intervals $(J_a)_{a\in F}$
of lengths $\pi(a)$, allowing empty intervals, with
$J_b=[0,p)$. Define $f(u)=a$ for $u\in J_a$.
Let $(Z_n)$ be independent uniform random variables on
$[0,1)$, and define
\[
 Y_n^{(j)}
 =f\left(\left(Z_n+\frac{j-1}{r}\right)\bmod1\right),
 \qquad 1\le j\le r.
\]
Each sequence $Y^{(j)}$ has distribution $\pi^{\otimes\mathbb N}$.

Since $p\ge1/r$, for every $n$ at least one of the $r$ equally spaced points belongs to $[0,1/r)\subseteq[0,p)$, so
\[
 \bigcup_{j=1}^r\{n:Y_n^{(j)}=b\}=\mathbb N.
\]
Almost surely all the sequences $Y^{(j)}$ belong to $H$,
and hence none of the sets in this finite union belongs
to $\mathcal U$. This contradicts the ultrafilter property.
\end{proof}

We shall also use the following standard property of
product measures; see \cite[Theorem~254Lb]{Fremlin2}.

\begin{lemma}
\label{bmfc:lem:product}
For each $j\ge1$, let $P_j$ be a Borel probability on a
compact metric space $\Omega_j$, and let
$X_j\subseteq\Omega_j$ have outer $P_j$-probability one.
Then $\prod_{j\ge1}X_j$ has outer probability one under
$\bigotimes_{j\ge1}P_j$.
\end{lemma}

\subsection*{Proof of the theorem}

\begin{proof}
\emph{Construction of the charges.}
Fix a nonprincipal ultrafilter $\mathcal U$.
Recall from Section~\ref{bm:sec:charges-ultrafilters}
that ultrafilter limits along $\mathcal U$ are linear and
monotone, and that a nonprincipal ultrafilter limit agrees
with ordinary limits whenever the latter exist.

If $(\lambda_n)$ is a sequence of charges, then the
pointwise ultrafilter limit
\[
 \lambda(E)=\lim_{n\to\mathcal U}\lambda_n(E),
 \qquad E\subseteq\mathbb N,
\]
is a charge, and for every bounded sequence $g$,
\[
 \lambda(g)=\lim_{n\to\mathcal U}\lambda_n(g).
\]
The same facts apply successively to iterated ultrafilter
limits, taken in the specified order.

For $1\le k\le n$, put
\[
 I_{k,n}
 =
 \left\{
 n^4+k(k-1)+1,\ldots,
 n^4+k^2
 \right\},
 \qquad
 D=\bigcup_{n\ge1}\bigcup_{k=1}^n I_{k,n}.
\]
These intervals are pairwise disjoint and have
$|I_{k,n}|=k$. For $m^4\le N<(m+1)^4$,
\[
 \frac{|D\cap\{1,\ldots,N\}|}{N}
 \le
 \frac{1}{m^4}\sum_{n=1}^m\sum_{k=1}^n k
 =
 \frac{m(m+1)(m+2)}{6m^4}
 \longrightarrow 0.
\]
Thus $D$ has density zero.

For every $E\subseteq\mathbb N$, define
\begin{equation}
 \begin{split}
 \nu_k(E)
   &=\lim_{\substack{n\to\mathcal U\\ n\ge k}}
         \frac{|E\cap I_{k,n}|}{k},
 \qquad
 \nu(E)=\lim_{k\to\mathcal U}\nu_k(E),\\
 \rho(E)
   &=\lim_{N\to\mathcal U}
         \frac{|E\cap\{1,\ldots,N\}|}{N}.
 \end{split}
 \label{bmfc:eq:charges}
\end{equation}
By the facts recalled above, $\nu_k$ and $\rho$ are charges,
and $\nu$, being an ultrafilter limit of the charges $\nu_k$,
is also a charge. In the definition of $\nu_k$, the restriction
$n\ge k$ deletes only finitely many indices; since $\mathcal U$
is nonprincipal, it does not affect the ultrafilter limit.

Since $I_{k,n}$ is an interval, for every $1\le k\le n$
and $E\subseteq\mathbb N$,
\[
 -1
 \le |(E+1)\cap I_{k,n}|-|E\cap I_{k,n}|
 \le 1.
\]
Dividing by $k$ and taking the ultrafilter limit in $n$ gives
\[
 |\nu_k(E+1)-\nu_k(E)|\le\frac1k.
\]
Since $1/k\to0$ as $k\to\infty$, taking the ultrafilter limit in $k$ gives
\[
 \nu(E+1)
 =\lim_{k\to\mathcal U}\nu_k(E+1)
 =\lim_{k\to\mathcal U}\nu_k(E)
 =\nu(E).
\]
Thus $\nu$ is translation-invariant. Similarly, since $I_N=\{1,\ldots,N\}$ is an interval, for every $N\ge1$ and $E\subseteq\mathbb N$,
\[
 -1
 \le |(E+1)\cap I_N|-|E\cap I_N|
 \le 1.
\]
Dividing by $N$ and taking the ultrafilter limit in $N$ gives $\rho(E+1)=\rho(E)$. Thus $\rho$ is translation-invariant. In addition, since $I_{k,n}\subseteq D$ for every $n\ge k$, we have $\nu_k(D)=1$ for every $k$, and hence $\nu(D)=1$. Since $D$ has density zero,
\[
 \rho(\mathbb N\setminus D)=1.
\]
Thus $\nu$ and $\rho$ are concentrated on disjoint subsets of $\mathbb N$.

The definition of $\rho$ also implies that it agrees
with ordinary Ces\`aro limits whenever they exist.

The same charges $\rho$ and $\nu$ will be used for all
strategy pairs and all $\alpha\in[0,1]$.

\medskip
\noindent\emph{A common set of pure sequences.}
Define $(X_k)_{k\ge1}$ and $X$ by
\begin{align}
 X_k  &=
 \left\{a\in\{L,R\}^{\mathbb N}:
       \{n\ge k:a_t=L\text{ for every }t\in I_{k,n}\}
          \in\mathcal U\right\},
 \notag\\
 X&=\bigcap_{k\ge1}X_k.
 \label{bmfc:eq:common-set}
\end{align}

For each $k\ge1$, put
\[
 D_k=\bigcup_{n\ge k}I_{k,n}.
\]
The condition defining $X_k$ depends only on
the coordinates indexed by $D_k$.

Recall that, for $0<p<1$, $P_p$ is the product probability
on $\{L,R\}^{\mathbb N}$ under which each coordinate equals
$L$ with probability $p$. We show that $X$ has outer $P_p$-probability one for every $p\in(0,1)$.

Under $P_p$, the blocks $(a_t)_{t\in I_{k,n}}$, $n\ge k$,
are independent and identically distributed on $\{L,R\}^k$, and the all-$L$ block has probability $p^k\in(0,1)$.
Since deleting the finitely many indices $n<k$ does not
affect membership in $\mathcal U$,
Lemma~\ref{bmfc:lem:ultrafilter} applies to these blocks.
It follows that the set of sequences
$a\in\{L,R\}^{D_k}$ for which
\[
 \{n\ge k:a_t=L\text{ for every }t\in I_{k,n}\} \in\mathcal U
\]
has outer probability one under the marginal of $P_p$
on $\{L,R\}^{D_k}$.

The sets $D_k$, $k\ge1$, are pairwise disjoint and their
union is $D$.
Applying Lemma~\ref{bmfc:lem:product} to the factors
$\{L,R\}^{D_k}$, together with the full factor
$\{L,R\}^{\mathbb N\setminus D}$, shows that $X$ has
outer $P_p$-probability one.

Write $L(a)=\{t:a_t=L\}$ and
$R(a)=\mathbb N\setminus L(a)$. Hence
\begin{align}
 a\in X
 &\quad\Longrightarrow\quad
 \nu_k(L(a))=1\quad\text{for every }k\ge1
 \notag\\
 &\quad\Longrightarrow\quad
 \nu(L(a))=1,\qquad \nu(R(a))=0.
 \label{bmfc:eq:common-mass}
\end{align}

Recall the notation $A_1(a)$ and $S(a)$ from
Proposition~\ref{prop:bm}.
The Borel set $\{a:A_1(a)+pS(a)\le p\}$ has positive $P_p$-probability. By the strong law of large numbers, the set of sequences with ordinary $L$-frequency $p$ has $P_p$-probability one. The intersection of these two
Borel sets therefore has positive $P_p$-probability and meets $X$, since $X$ has outer $P_p$-probability one.

We have thus proved that, for every $\sigma$ and every $0<p<1$, there is a pure sequence $a\in X\cap\{a:A_1(a)+pS(a)\le p\}$ whose ordinary
$L$-frequency is $p$. Since $\rho$ agrees with ordinary Ces\`aro limits and $a\in X$, it follows that
\begin{equation}
 \rho(R(a))=1-p,\qquad
 \nu(R(a))=0,\qquad
 A_1(a)+pS(a)\le p.
 \label{bmfc:eq:selection}
\end{equation}

\medskip
\noindent\emph{Upper bound on the lower value.}
For $0<\alpha\le1$, set
\[
 c=\frac{\alpha}{1+\alpha}.
\]
Apply \eqref{bmfc:eq:selection} with $p=c$.
Since $\alpha(1-c)=c$, the selected pure sequence
satisfies
\[
 \mu_\alpha(R(a))
 =\alpha\rho(R(a))+(1-\alpha)\nu(R(a))
 =\alpha(1-c)=c.
\]
Equations \eqref{bmfc:eq:pure-payoff} and
\eqref{bmfc:eq:selection} give
\[
 \mu_\alpha(\sigma,a)
 =\alpha\rho(\sigma,a)+(1-\alpha)\nu(\sigma,a)
 =A_1(a)+cS(a)\le c.
\]

The assertion for $\alpha=0$ follows by continuity. Indeed,
\[
 |\mu_\alpha(\sigma,a)-\nu(\sigma,a)|
 =\alpha|\rho(\sigma,a)-\nu(\sigma,a)|
 \le\alpha
\]
uniformly in $\sigma$ and $a$. Hence, for every fixed
$\sigma$,
\[
 \inf_{a\ \mathrm{pure}}\nu(\sigma,a)
 =
 \lim_{\alpha\downarrow0}
 \inf_{a\ \mathrm{pure}}\mu_\alpha(\sigma,a).
\]
Together with nonnegativity and the bounds just proved,
the right-hand side is zero. Thus
\[
 \inf_{a\ \mathrm{pure}}\nu(\sigma,a)=0.
\]
This proves all the asserted pure-response bounds.

\medskip
\noindent\emph{Lower bound on the lower value.}
The case $\alpha=0$ follows from nonnegativity.
For $\alpha>0$, consider the modified Big Match
\[
 \begin{array}{c|cc}
       &L&R\\ \hline
 T     &0&\alpha\\
 B     &1^*&0^*
 \end{array}.
\]
Its normalized $\beta$-discounted value is
$c=\alpha/(1+\alpha)$ for every $\beta\in(0,1]$.
Indeed, the Shapley matrix at continuation value $c$ is
\[
 \begin{pmatrix}
 (1-\beta)c & \beta\alpha+(1-\beta)c\\
 1&0
 \end{pmatrix}.
\]
Player~2 assigns probability $c$ to $L$, and Player~1
assigns probability $\beta\alpha/(1+\beta\alpha)$ to $B$.
These mixed actions each yield $c$ against either
pure action of the opponent.
Thus the $\beta$-discounted value is $c$ for every
$\beta\in(0,1]$.
By Kohlberg's theorem for absorbing games~\cite{Koh74}, the game has value $c$; moreover, for every $\varepsilon>0$ there is a strategy of Player~1
whose expected average payoff is at least $c-\varepsilon$ against every strategy of Player~2 for all sufficiently long horizons.

For arbitrary strategies $(\sigma,\tau)$, let $g_t$ be the expected stage payoff in the original game, and let $A_t$ be the probability of absorption at payoff one by stage $t$.
Write $A_\infty=\lim_t A_t$.
Since $\nu$ and $\rho$ are diffuse, they agree with
ordinary limits on convergent sequences; hence
\[
 \nu(A)=\rho(A)=A_\infty.
\]
Since $g_t\ge A_t$, monotonicity gives $\nu(g)\ge A_\infty$.
Under the same strategies, the expected payoff in
the modified game is
\[
 \widetilde g_t=\alpha g_t+(1-\alpha)A_t.
\]
Consequently,
\begin{equation}
 \mu_\alpha(g)
  \ge\alpha\rho(g)+(1-\alpha)A_\infty
  =\rho(\widetilde g).
 \label{bmfc:eq:comparison}
\end{equation}

Choose the strategy supplied by Kohlberg's theorem for the modified game. Since $\rho$ is an ultrafilter limit of initial averages, this strategy satisfies
\[
 \rho(\widetilde g)\ge c-\varepsilon
\]
against every $\tau$.
Equation \eqref{bmfc:eq:comparison} then gives the
same guarantee under $\mu_\alpha$ in the original
game. Letting $\varepsilon$ tend to zero proves
$\underline v(\mu_\alpha)\ge c$.

Together with the pure-response bounds, this proves
\eqref{bmfc:eq:values}. Finally, for any $v\in[0,1/2]$, taking
$\alpha=v/(1-v)$ gives $\underline v(\mu_\alpha)=v$.

This completes the proof of Theorem~\ref{thm:bm-main}.
\end{proof}

\section{Stochastic games under translation-invariant charges}

We conclude with a research program concerning general stochastic
games under translation-invariant charges. We formulate it from
the point of view of Player~1; the corresponding statements for
Player~2 follow by duality.

Let $\mathcal I$ denote the set of translation-invariant charges
on $\mathbb N$. For an initial state $s$, define the \emph{robust
guarantee}
\[
 V^{\mathrm{rob}}(s)
 :=
 \sup_{\sigma}\inf_{\tau}\inf_{\mu\in\mathcal I}
       \mu(\sigma,\tau)(s),
\]
and the \emph{fixed-charge guarantee}
\[
 V^{\mathrm{fix}}(s)
 :=
 \inf_{\mu\in\mathcal I}
       \sup_{\sigma}\inf_{\tau}\mu(\sigma,\tau)(s).
\]
Thus, in the first expression Player~1 chooses a strategy before
the translation-invariant charge is selected, whereas in the
second expression the charge is fixed before Player~1 chooses
his strategy. Hence
\[
 V^{\mathrm{rob}}(s)\le V^{\mathrm{fix}}(s).
\]
Moreover, by Mertens and Neyman, the stochastic game with initial
state $s$ has value $v(s)$. In particular, $v(s)$ is the value
under every patient valuation. Since every patient valuation is
a translation-invariant charge,
\[
 V^{\mathrm{fix}}(s)\le v(s).
\]

For a stationary strategy $x$ of Player~1, let $w(x)(s)$ denote
the limiting-average value, with initial state $s$, of the MDP
faced by Player~2 when Player~1 uses $x$, with Player~2 minimizing.
By Remark~\ref{rem:stationary}, for every
$\mu\in\mathcal I$,
\[
 \inf_{\tau}\mu(x,\tau)(s)
 =
 \inf_{\tau}\inf_{\mu'\in\mathcal I}
       \mu'(x,\tau)(s)
 =
 w(x)(s).
\]
Thus the guarantee of a stationary strategy is independent of
the translation-invariant charge: its guarantee under each fixed
translation-invariant charge and its guarantee simultaneously over all
translation-invariant charges both equal $w(x)(s)$. Define
\[
 V^{\mathrm{stat}}(s):=\sup_x w(x)(s),
\]
where the supremum ranges over stationary strategies of Player~1.
Consequently,
\[
 V^{\mathrm{stat}}(s)\le V^{\mathrm{rob}}(s).
\]

We therefore have, for every initial state $s$,
\begin{equation}\label{eqn:inequalities-chain}
 V^{\mathrm{stat}}(s)
 \le V^{\mathrm{rob}}(s)
 \le V^{\mathrm{fix}}(s)
 \le v(s).
\end{equation}

For an MDP, viewed as a stochastic game with only one
nontrivial player, all these inequalities are equalities by
Theorem~\ref{thm:mdp}. In the Big Match, with its nonabsorbing
state as the initial state, Proposition~\ref{prop:bm} and
Theorem~\ref{thm:bm-main} show that
\[
 V^{\mathrm{stat}}
 =
 V^{\mathrm{rob}}
 =
 V^{\mathrm{fix}}
 =
 0
 <
 \frac12
 =
 v.
\]
The first objective is to characterize, for each of the
inequalities in \eqref{eqn:inequalities-chain}, the stochastic
games and initial states for which equality holds and those
for which the inequality is strict. 

For a fixed $\mu\in\mathcal I$, let
\[
 \underline v_\mu(s)
 :=
 \sup_\sigma\inf_\tau\mu(\sigma,\tau)(s),
\]
and
\[
 \underline{\mathcal V}(s)
 :=
 \{\underline v_\mu(s):\mu\in\mathcal I\}.
\]

A second, finer objective is to determine the set
$\underline{\mathcal V}(s)$ of lower values generated by
translation-invariant charges. In particular, when one or
more of the inequalities in \eqref{eqn:inequalities-chain}
are strict, which values between the corresponding benchmarks
belong to $\underline{\mathcal V}(s)$?

For the Big Match, Theorem~\ref{thm:bm-main} shows moreover that
\[
 \underline{\mathcal V}(s)=[0,1/2].
\]

\paragraph{Acknowledgment of AI assistance.}
The author used OpenAI's ChatGPT, including GPT-5.6, and Anthropic's Claude in the development of this paper. The assistance included discussion of mathematical arguments and possible extensions, checking proofs and counterexamples, and suggestions concerning exposition and organization. The author critically reviewed and independently verified the mathematical content and takes full responsibility for the paper. The assistance was sufficiently substantial that, had comparable contributions been made by a human collaborator, the author would have regarded them as warranting coauthorship.

\end{document}